\documentclass{amsart}
\usepackage{amssymb,amstext,amsmath,amscd,amsthm,amsfonts,enumerate,graphicx,latexsym,stmaryrd,multicol}

\usepackage[usenames]{color}
\usepackage[all]{xy}
\newtheorem{thm}{Theorem}[section]
\newtheorem{lem}[thm]{Lemma}
\newtheorem{prop}[thm]{Proposition}
\newtheorem{cor}[thm]{Corollary}
\theoremstyle{definition}
\newtheorem{dfn}[thm]{Definition}
\newtheorem{ques}[thm]{Question}

\newtheorem{eg}[thm]{Example}

\newtheorem{rmk}[thm]{Remark}
\theoremstyle{remark}

\newtheorem*{ac}{Acknowlegments}
\newtheorem*{conv}{Convention}
\numberwithin{equation}{thm}
\def\add{\mathsf{add}}
\def\Ann{\mathsf{Ann}}

\def\assh{\operatorname{Assh}}
\def\C{\mathsf{C}}
\def\ca{\mathsf{ca}}
\def\cm{\mathsf{CM}}

\def\Deep{\mathsf{Deep}}
\def\depth{\operatorname{depth}}
\def\dim{\operatorname{dim}}
\def\ext{\operatorname{ext}}
\def\Ext{\operatorname{Ext}}
\def\fl{\operatorname{fl}}
\def\G{\mathsf{G}}

\def\inf{\operatorname{inf}}

\def\m{\mathfrak{m}}
\def\mod{\operatorname{mod}}

\def\NF{\mathrm{NF}}

\def\p{\mathfrak{p}}

\def\q{\mathfrak{q}}
\def\radius{\operatorname{radius}}
\def\res{\operatorname{res}}
\def\S{\mathsf{S}}

\def\Sing{\operatorname{Sing}}
\def\spec{\operatorname{Spec}}
\def\sup{\operatorname{sup}}
\def\supp{\operatorname{Supp}}

\def\Tor{\operatorname{Tor}}

\def\V{\mathrm{V}}
\def\X{\mathcal{X}}

\def\Y{\mathcal{Y}}

\begin{document}
\allowdisplaybreaks
\title[A generalization of the dimension and radius of a subcategory of modules]{A generalization of the dimension and radius of a subcategory of modules and its applications}
\author{Yuki Mifune}
\address{Graduate School of Mathematics, Nagoya University, Furocho, Chikusaku, Nagoya 464-8602, Japan}
\email{yuki.mifune.c9@math.nagoya-u.ac.jp}
\begin{abstract}
Let $R$ be a commutative noetherian local ring and denote by $\mod R$ the category of finitely generated $R$-modules. In this paper, we give some evaluations of the singular locus of $R$ and annihilators of Tor and Ext from a viewpoint of the finiteness of dimensions/radii of full subcategories of $\mod R$. As an application, we recover a theorem of Dey and Takahashi when $R$ is Cohen--Macaulay. Moreover, we obtain the divergence of the dimensions of specific full subcategories of $\mod R$ in non-Cohen--Macaulay case.
\end{abstract}
\maketitle
%\tableofcontents
%%%%%%%%%%%%%%%%%%%%%%%%%%%%%%%%%%%%%%%%%%%%%%%%%%%%%%%%%%%%
%%%%%%%%%%%%%%%%%%%%%%%%%%%%%%%%%%%%%%%%%%%%%%%%%%%%%%%%%%%%
\section{Introduction}
Let $R$ be a commutative noetherian local ring. Denote by $\mod R$ the category of finitely generated $R$-modules. The notions of the dimension and radius of a full subcategory of $\mod R$ have been introduced by Dao and Takahashi \cite{Dao Takahashi dimension, Dao Takahashi radius}. These notions are abelian analogs of the dimension of a triangulated category introduced by Rouquier \cite{Rouquier}, and it has turned out that they can connect several notions appearing in commutative algebra and representation theory.
\par
The concept of a resolving subcategory has been introduced by Auslander and Bridger \cite{AB} in the 1960s and has been studied widely and deeply so far; see \cite{APST, AR, KS, Takahashi 2021, stcm, arg} for instance. Besides, the finiteness of dimension/radius of a resolving subcategory of $\mod R$ also has been studied; see \cite{Dao Takahashi dimension, Dao Takahashi radius, M, Sadeghi Takahashi} for example. Roughly speaking, the dimension/radius of a full subcategory of $\mod R$ is the least number of extensions to build up the subcategory from one object, up to finite direct sums, direct summands, and syzygies. (The precise definition will be recalled in Definition \ref{def of rad}.)
\par
For full subcategories $\X, \Y$ of $\mod R$, we denote by $\Ann \Tor(\X,\Y)$ (respectively, $\Ann \Ext(\X,\Y)$) the intersection of ideals $\Ann_R \Tor_i ^R(X,Y)$ (respectively, $\Ann_R \Ext_R ^i (X,Y)$) of all positive integers $i$ and objects $X \in \X$, $Y \in \Y$. 
When $R$ is Cohen--Macaulay, the relationships among $\cm_i(R)$, $\Ann \Tor(\X,\Y)$, $\Ann \Ext(\X,\Y)$, and the singular locus of $R$ have been studied by Dey and Takahashi \cite{Dey Takahashi}.
In this paper we shall unify the notions of dimension and radius. 
To be more precise, we introduce the {\em radius of $\X$ in $\Y$}, denoted $\radius(\X, \Y)$, which measures how many extensions are necessary to build $\X$ up in $\Y$ up to finite direct sums, direct summands, and syzygies. 
The main result of this paper is the following theorem, which extends a theorem of Dey and Takahashi \cite{Dey Takahashi} to the non-Cohen--Macaulay case. (For more details, see Remark \ref{rmk of cor2}.)
%%%%%%%%%%%%%%%%%%%%%%%%%%%%%%%%%%%%%%%%%%%%%%%%%%%%%%%%%%%%
\begin{thm}[Theorem \ref{main thm}]
Let $R$ be a noetherian local ring with residue field $k$ and $\X, \Y$ be full subcategories of $\mod R$ with $\radius(\X,\Y) < \infty$. Assume that $\X$ is closed under extensions and contains $\Omega_R ^n k$ for some $n\geq0$. 
Then both of the closed subsets $\V(\Ann \Tor(\X,\Y))$ and $\V(\Ann \Ext(\X,\Y))$ of $\spec R$ contain the singular locus of $R$ and contained in the nonfree locus of $\Y$. 
\end{thm}
%%%%%%%%%%%%%%%%%%%%%%%%%%%%%%%%%%%%%%%%%%%%%%%%%%%%%%%%%%%%
Denote by $\C(R)$ the full subcategory of $\mod R$ consisting of $R$-modules $M$ such that the inequality $\depth M_{\p} \geq \depth R_{\p}$ holds for all prime ideals $\p$ of $R$. 
This resolving subcategory has been introduced by Takahashi \cite{Takahashi 2023}, and it coincides with the full subcategory $\cm(R)$ of maximal Cohen--Macaulay $R$-modules when $R$ is Cohen--Macaulay. 
As an application of the above theorem, we relate the finiteness of dimensions/radii of specific resolving subcategories to the dimensions of the closed subsets defined by the annihilators of $\Tor$ and $\Ext$, and singular locus of $R$.
%%%%%%%%%%%%%%%%%%%%%%%%%%%%%%%%%%%%%%%%%%%%%%%%%%%%%%%%%%%%
\begin{cor} [Corollary \ref{cor1 of thm}]
Let $R$ be a noetherian local ring and $n$ a nonnegative integer. 
Assume that  $\radius(\C_0 (R), \mod_n R) < \infty$. 
Then the dimensions of the closed subsets
$\V(\Ann \Tor(\C_n(R), \C_n(R)))$, 
$\V(\Ann \Ext(\C_n(R), \C_n(R)))$, and
the singular locus of $R$
are less than or equal to $n$.
\end{cor}
%%%%%%%%%%%%%%%%%%%%%%%%%%%%%%%%%%%%%%%%%%%%%%%%%%%%%%%%%%%%
Here $\C_n(R)$ is the full subcategory of $\C(R)$ consisting of modules $M$ such that the dimension of nonfree locus of $M$ is less than or equal to $n$. 
As a consequence of the above corollary, we obtain the divergence of the dimension of $\C_i(R)$ for $i=0,\cdots,\dim R-1$ when $R$ is non-Cohen--Macaulay.
The following result is included in Corollary \ref{cor2 of thm}.
%%%%%%%%%%%%%%%%%%%%%%%%%%%%%%%%%%%%%%%%%%%%%%%%%%%%%%%%%%%%
\begin{cor}
Let $R$ be a $d$-dimensional non-Cohen--Macaulay local ring. 
Then the dimension of $\C_i(R)$ is infinite for all $0\leq i \leq d-1$.
\end{cor}
%%%%%%%%%%%%%%%%%%%%%%%%%%%%%%%%%%%%%%%%%%%%%%%%%%%%%%%%%%%%
\par
The organization of this paper is as follows. In section 2 we state basic definitions and their properties for later use. In section 3 we give the proofs of the theorem and corollaries stated above.
%%%%%%%%%%%%%%%%%%%%%%%%%%%%%%%%%%%%%%%%%%%%%%%%%%%%%%%%%%%%
\begin{conv}
Throughout this paper, all rings are commutative noetherian local rings, all modules are finitely generated, and all subcategories are full and strict, that is, closed under isomorphism. Let $R$ be a (commutative noetherian local) ring with maximal ideal $\m$ and residue field $k$.
\end{conv}
%%%%%%%%%%%%%%%%%%%%%%%%%%%%%%%%%%%%%%%%%%%%%%%%%%%%%%%%%%%%%%%%%%%%
%%%%%%%%%%%%%%%%%%%%%%%%%%%%%%%%%%%%%%%%%%%%%%%%%%%%%%%%%%%%
\section{Preliminaries}
In this section, we recall various definitions and basic properties needed later. We begin with the notion of additive closures and syzygies of modules.
%%%%%%%%%%%%%%%%%%%%%%%%%%%%%%%%%%%%%%%%%%%%%%%%%%%%%%%%%%%%
\begin{dfn}
Let $\X$ be a subcategory of $\mod R$ and $M$ an $R$-module. 
\begin{enumerate}[\rm(1)] 
\item
We denote by $\add_R \X$ the {\em additive closure of $\X$}, namely, the subcategory of $\mod R$ consisting of direct summands of finite direct sums of modules in $\X$. When $\X$ consists of a single module $X$, we simply denote $\add_R \{ X\}$ by $\add_R X$. Note that $\add _R R$ consists of (finitely generated) free $R$-modules.
\item
Take a minimal free resolution $\cdots \xrightarrow{\delta_{n+1}} F_n \xrightarrow{\delta_n} F_{n-1} \xrightarrow{\delta_{n-1}}\cdots \xrightarrow{\delta_1}F_0 \xrightarrow{\delta_0} M \to 0$ of $M$. Then for each $n\geq 0$, the image of $\delta_n$ is called the {\em $n$-th syzygy of $M$} and denoted by $\Omega_R ^n M$.
\item
For a nonnegative integer $n$, we denote by $\Omega_R ^n \X$ the subcategory of $\mod R$ consisting of $R$-modules $M$ such  that there exists an exact sequence of the form $0 \to M \to F_{n-1} \to \cdots \to F_0 \to X \to 0 $ where $X \in \X$ and $F_i \in \add_R R$ for all $i \geq 0$. 
\end{enumerate}
\end{dfn}
%%%%%%%%%%%%%%%%%%%%%%%%%%%%%%%%%%%%%%%%%%%%%%%%%%%%%%%%%%%%
\begin{rmk}
\begin{enumerate}[\rm(1)]
\item
Let $\X, \Y$ be subcategories of $\mod R$ and $n$ a nonnegative integer. 
If one has $\X \subseteq \Y$, then one has 
$\add_R \X \subseteq \add_R \Y$, and 
$\Omega_R ^n \X \subseteq \Omega_R ^n \Y$.
\item
Let $\X$ be a subcategory of $\mod R$ and $n$ a nonnegative integer. Then one has
\begin{equation}
\Omega_R ^n \X = \nonumber
\begin{cases}
\X & \text{if $n=0$,} \\
\{\Omega_R ^n X \oplus R^{\oplus m} \mid X\in \X, m\geq 0 \}& \text{if $n>0$.} 
\end{cases}
\end{equation}
In particular, if $\X$ contains $0$ and $n>0$, then $\Omega_R ^n \X$ contains free $R$-modules. 
Note that both $\Omega_R ^n \{X\}$ and $\Omega_R ^n X$ are not equal for any $R$-modules $X$ in general.
\end{enumerate}
\end{rmk}
%%%%%%%%%%%%%%%%%%%%%%%%%%%%%%%%%%%%%%%%%%%%%%%%%%%%%%%%%%%%
Now we recall the definitions of a resolving subcategory and the resolving/extension closure of a subcategory of $\mod R$.
%%%%%%%%%%%%%%%%%%%%%%%%%%%%%%%%%%%%%%%%%%%%%%%%%%%%%%%%%%%%
\begin{dfn}
\begin{enumerate}[\rm(1)]
\item
A subcategory $\X$ of $\mod R$ is {\em resolving} if $\X$ contains $R$ and is closed under direct summands, extensions, and kernels of epimorphisms in $\mod R$. The last two closure properties say that for any exact sequence $0 \to L \to M \to N \to 0$ of $R$-modules with $N\in \X$, we have the equivalence $L \in \X \iff M \in \X$.
\item
For a subcategory $\X$ of $\mod R$, we denote by $\res \X$ the {\em resolving closure} of $\X$, namely, the smallest resolving subcategory of $\mod R$ containing $\X$. When $\X$ consists of a single module $X$, we simply denote $\res \{X\}$ by $\res X$.
\item
For a subcategory $\X$ of $\mod R$, we denote by $\ext \X$ the {\em extension closure} of $\X$, namely, the smallest subcategory of $\mod R$ which contains $\X$ and is closed under direct summands, and extensions in $\mod R$. When $\X$ consists of a single module $X$, we simply denote $\ext \{X\}$ by $\ext X$.
\end{enumerate}
\end{dfn}
%%%%%%%%%%%%%%%%%%%%%%%%%%%%%%%%%%%%%%%%%%%%%%%%%%%%%%%%%%%%
\begin{rmk} 
Let $R$ be a local ring with residue field $k$.
\begin{enumerate}[\rm(1)]
\item
Let $\X$ be a resolving subcategory of $\mod R$. Then $\X$ contains (finitely generated) free $R$-modules and closed under syzygies.
\item
Denote by $\fl R$ the subcategory of $\mod R$ consisting of modules of finite length. Since any $R$-module of finite length has composition series, one has $\fl R = \ext k$.
\end{enumerate}
\end{rmk}
%%%%%%%%%%%%%%%%%%%%%%%%%%%%%%%%%%%%%%%%%%%%%%%%%%%%%%%%%%%%
Next, we state the definition of the radius of a subcategory in another subcategory. This is a common generalization of the notions of the dimension and radius of a subcategory introduced by Dao and Takahashi \cite{Dao Takahashi dimension, Dao Takahashi radius}.
%%%%%%%%%%%%%%%%%%%%%%%%%%%%%%%%%%%%%%%%%%%%%%%%%%%%%%%%%%%%
\begin{dfn} \label{def of rad}
Let $\X, \Y$ be subcategories of $\mod R$ and $M$ an $R$-module.
\begin{enumerate}[\rm(1)]
\item
We denote by $\lbrack \X \rbrack$ the additive closure of the subcategory consisting of $R$ and all modules of the form $\Omega_R ^i X$, where $i\geq0$ and $X\in \X$. When $\X$ consists of a single module $X$, we simply denote $\lbrack \{X\} \rbrack$ by $\lbrack X \rbrack$. 
\item
We denote by $\X \circ \Y$ the subcategory of $\mod R$ consisting of $R$-modules $Z$ such that there exists an exact sequence $0 \to X \to Z \to Y \to 0$ with $X\in \X$ and $Y\in \Y$. We set $\X \bullet \Y = \lbrack\lbrack \X \rbrack \circ \lbrack \Y \rbrack \rbrack$. 
\item
We inductively define the {\em ball of radius $r$ centered at $\X$} as follows.
\begin{equation}
\lbrack\X\rbrack_r = \nonumber
\begin{cases}
\lbrack\X\rbrack & \text{if $r=1$,} \\
\lbrack\X\rbrack_{r-1} \bullet \lbrack\X\rbrack & \text{if $r>1$.}
\end{cases}
\end{equation}
For convention, we set $\lbrack\X\rbrack_0 = 0$.
If $\X$ consists of a single module $G$, then we simply denote $\lbrack\{G\}\rbrack_r$ by $\lbrack G\rbrack_r$ and call it the {\em ball of radius} $r$ centered at $G$.
Note that the operator “$\bullet$” is associative; see \cite[Proposition 2.2]{Dao Takahashi radius}. 
\item
We define the {\em radius} of $\X$ in $\Y$, denoted by $\radius(\X,\Y)$, as the infimum of the nonnegative integers $n$ such that there exists a ball of radius $n+1$ centered at a module containing $\X$ and contained in $\Y$, that is, 
\begin{equation*} 
\radius(\X,\Y)=\inf\{n \geq 0\mid \X \subseteq[G]_{n+1} \subseteq \Y \text{ for some }G\in\mod R\}.
\end{equation*}
\end{enumerate}
\end{dfn}
%%%%%%%%%%%%%%%%%%%%%%%%%%%%%%%%%%%%%%%%%%%%%%%%%%%%%%%%%%%%
\begin{rmk}
Let $\X, \Y$ be subcategories of $\mod R$.
\begin{enumerate}[\rm(1)]
\item
The notion of dimensions and radii of subcategories introduced by Dao and Takahashi \cite{Dao Takahashi dimension, Dao Takahashi radius} are expressed as follows.
\begin{itemize}
\item
$\dim \X = \radius(\X, \X)=\inf\{n \geq0 \mid\ X = [G]_{n+1} \text{ for some }G\in \X \} $.
\item
$\radius \X =\radius(\X, \mod R)=\inf \{n\geq0\mid\X\subseteq[G]_{n+1} \text{ for some }G\in\mod R\}$.
\end{itemize}
\item
Let $\X'$ be another subcategory of $\mod R$. If one has $\X \subseteq \X'$, then one has $\radius(\X, \Y) \leq \radius(\X', \Y)$.
\item
Let $\Y'$ be another subcategory of $\mod R$. If one has $\Y \subseteq \Y'$, then one has $\radius(\X, \Y) \geq \radius(\X, \Y')$.
\item
Let $\X, \Y$ be subcategories of $\mod R$. 
If $\Y$ is resolving, then we have 
\begin{center}
$\radius(\X,\Y)=\inf\{n \geq 0\mid \X \subseteq[G]_{n+1} \text{ for some }G\in \Y \}$.
\end{center}
\end{enumerate}
\end{rmk}
%%%%%%%%%%%%%%%%%%%%%%%%%%%%%%%%%%%%%%%%%%%%%%%%%%%%%%%%%%%%
Recall that a subset $W$ of $\spec R$ is {\em specialization-closed} if $\p \subseteq \q$ are prime ideals of $R$ and $W$ contains $\p$, then $\q$ belongs to $W$.
%%%%%%%%%%%%%%%%%%%%%%%%%%%%%%%%%%%%%%%%%%%%%%%%%%%%%%%%%%%%
\begin{dfn}
\begin{enumerate}[\rm(1)]
\item
Let $\Phi$ be a subset of $\spec R$. The {\em dimension} of $\Phi$, denoted by $\dim \Phi$, is defined as the supremum of nonnegative integers $n$ such that there exists a chain $\p_0 \subsetneq \p_1 \subsetneq \cdots \subsetneq \p_n$ of prime ideals in $\Phi$.
\item
We denote by $\Sing(R)$ the {\em singular locus} of $R$, namely, the set of prime ideals $\p$ of $R$ such that $R_{\p}$ is not regular. We say that $R$ is an {\em isolated singularity} if $\Sing(R) \subseteq \{\m \}$, or equivalently, if $\Sing(R)$ has dimension at most zero. Note that $\Sing(R)$ is a specialization-closed subset of $\spec R$.
\item
Let $\X$ be a subcategory of $\mod R$ and $M$ an $R$-module.
Denote by $\NF (M)$ the {\em nonfree locus} of M, namely, the set of prime ideals $\p$ of $R$ such that $M_{\p}$ is not a free $R_{\p}$-module.
In addition, the {\em nonfree locus} of  $\X$ is defined by 
$\NF(\X) = \displaystyle\bigcup _{X\in \X} \NF(X)$.
\end{enumerate}
\end{dfn}
%%%%%%%%%%%%%%%%%%%%%%%%%%%%%%%%%%%%%%%%%%%%%%%%%%%%%%%%%%%%
\begin{rmk}
\begin{enumerate}[\rm(1)]
\item
A subset $W$ of $\spec R$ is specialization-closed if and only if W is a union of closed subsets of $\spec R$ in the Zariski topology.
\item
If $W$ is a specialization-closed subset of $\spec R$, then one has $\dim W = \sup \{\dim R/\p \mid \p \in W \}$.
\item
Let $\X \subseteq \Y$ be subcategories of $\mod R$. Then one has $\NF(\X) \subseteq \NF(\Y)$.
\item
Let $M$ be an $R$-module. 
Then one has $\NF(M) = \supp \Ext_R ^1(M, \Omega_R M)$ by \cite[Proposition 2.10]{Takahashi 2009}. Hence $\NF(M)$ is a closed subset of $\spec R$.
\item
Let $\X$ be a subcategory of $\mod R$. 
Then for any positive integer $n$, the equalities
\begin{center}
$\NF(\X) = \NF([\X]_n) = \NF(\res \X)$
\end{center}
hold. This follows from \cite[Corollary 3.6]{Takahashi 2009}.
\item
Let $n$ be a nonnegative integer and $W$ a specialization-closed subset of $\spec R$. Let $W_n$ be the set of prime ideals $\p$ of $R$ such that $\dim R/\p \leq n$. Then one has $W \subseteq W_n \iff \dim W \leq n.$
\end{enumerate}
\end{rmk}
%%%%%%%%%%%%%%%%%%%%%%%%%%%%%%%%%%%%%%%%%%%%%%%%%%%%%%%%%%%%
We close this section by defining some specific subcategories of $\mod R$, and state basic properties of them. Notation is based on \cite{Dao Eghbali Lyle, Takahashi 2023}.
%%%%%%%%%%%%%%%%%%%%%%%%%%%%%%%%%%%%%%%%%%%%%%%%%%%%%%%%%%%%
\begin{dfn}
Let $n$ be a nonnegative integer.
\begin{enumerate}[\rm(1)]
\item
Denote by $\mod_n R$ the subcategory of $\mod R$ consisting of $R$-modules $M$ such that the inequality $\dim \NF(M) \leq n$ holds.
\item
Denote by $\Deep (R)$ the subcategory of $\mod R$ consisting of $R$-modules $M$ such that the inequality $\depth M \geq \depth R$ holds. 
We set $\Deep_n (R) = \Deep (R) \cap \mod_n R$.
\item
Denote by $\C(R)$ the subcategory of $\mod R$ consisting of $R$-modules such that the inequality $\depth M_{\p} \geq \depth R_{\p}$ holds for all prime ideals $\p$ of $R$.
We set $\C_n (R) = \C (R) \cap \mod_n R$.
\end{enumerate}
\end{dfn}
%%%%%%%%%%%%%%%%%%%%%%%%%%%%%%%%%%%%%%%%%%%%%%%%%%%%%%%%%%%%
\begin{rmk}
Let $R$ be a $d$-dimensional local ring.
\begin{enumerate}[\rm(1)]
\item
If $R$ is Cohen--Macaulay, then the subcategories $\Deep (R)$ and $\C(R)$ coincide with the category $\cm(R)$ of maximal Cohen--Macaulay $R$-modules.
\item
All the subcategories $\mod_n R$, $\Deep (R)$, $\Deep_n (R)$, $\C(R)$, $\C_n (R)$ are resolving subcategories of $\mod R$.
\item
Since the modules in $\mod_0 R$ are locally free on the punctured spectrum of $R$, the equality $\Deep_0 (R) = \C_0 (R)$ holds. The relationships among subcategories defined above are as follows. 
\begin{equation*}
\xymatrix@C=4pt@R=6pt{
\mod_0 R & \subseteq &  \mod_1 R & \subseteq & \cdots & \subseteq & \mod_d R & = & \mod R \\
\rotatebox{90}{$\subseteq$} & & \rotatebox{90}{$\subseteq$} & & &  & \rotatebox{90}{$\subseteq$}& & \rotatebox{90}{$\subseteq$} \\
\Deep_0 (R) & \subseteq & \Deep_1 (R) & \subseteq & \cdots & \subseteq & \Deep_d (R) & = & \Deep(R) \\
\rotatebox{90}{$=$} & & \rotatebox{90}{$\subseteq$} & & &  & \rotatebox{90}{$\subseteq$}& & \rotatebox{90}{$\subseteq$} \\
\C_0 (R) & \subseteq & \C_1(R) & \subseteq & \cdots & \subseteq & \C_d(R) & = & \C(R)
}
\end{equation*}
\end{enumerate}
\end{rmk}
%%%%%%%%%%%%%%%%%%%%%%%%%%%%%%%%%%%%%%%%%%%%%%%%%%%%%%%%%%%%%%%%%%%%%
%%%%%%%%%%%%%%%%%%%%%%%%%%%%%%%%%%%%%%%%%%%%%%%%%%%%%%%%%%%%%%%%%%%%%
\section{Proof of main theorem}
The main purpose of this section is to state and prove Theorem \ref{main thm}. First of all, we introduce some notation for convenience.
%%%%%%%%%%%%%%%%%%%%%%%%%%%%%%%%%%%%%%%%%%%%%%%%%%%%%%%%%%%%
\begin{dfn}
Let $\X, \Y$ be subcategories of $\mod R$.
We define $\Tor(\X,\Y)$ and $\Ext(\X,\Y)$ by
$$
\Tor(\X,\Y)=\bigoplus_{X\in\X,Y\in\Y,i>0}\Tor_i^R(X,Y),\qquad
\Ext(\X,\Y)=\bigoplus_{X\in\X,Y\in\Y,i>0}\Ext_R^i(X,Y).
$$
When $\X$ (respectively, $\Y$) consists of a single module $X$ (respectively, $Y$), we simply denote them by $\Tor(X,\Y)$ and $\Ext(X,\Y)$ (respectively, $\Tor(\X,Y)$ and $\Ext(\X,Y)$).
\end{dfn}
%%%%%%%%%%%%%%%%%%%%%%%%%%%%%%%%%%%%%%%%%%%%%%%%%%%%%%%%%%%%
\begin{rmk}
Let $\X \subseteq \X'$ and $\Y \subseteq \Y'$ be subcategories of $\mod R$. Then the following inclusion relations hold.
\begin{itemize}
\item
$\V (\Ann \Tor(\X,\Y)) \subseteq \V (\Ann \Tor(\X',\Y')) $.
\item
$\V (\Ann \Ext(\X,\Y)) \subseteq \V (\Ann \Ext(\X',\Y')) $.
\end{itemize}
\end{rmk}
%%%%%%%%%%%%%%%%%%%%%%%%%%%%%%%%%%%%%%%%%%%%%%
We need the following lemma which is proved by Dao and Takahashi.
%%%%%%%%%%%%%%%%%%%%%%%%%%%%%%%%%%%%%%%%%%%%%%
\begin{lem}\cite[Corollary 4.4]{Dao Takahashi dimension} \label{4.4}
Let $R$ be a $d$-dimensional local ring and $a\in R$, $n\geq0$. Then the following hold.
\begin{enumerate}[\rm(1)]
\item
Suppose that 
$a\Tor_i ^R(\fl R, \fl R) = 0$ for all $n-2d\leq i\leq n$. 
Then one has $a^{2^{2d}}\Tor_n ^R(\mod R,\mod R) = 0$.
\item
Suppose that 
$a\Ext_R ^i(\fl R, \fl R) = 0$ for all $n\leq i\leq n+2d$. 
Then one has $a^{2^{2d}}\Ext_R ^n(\mod R,\mod R) = 0$.
\end{enumerate}
\end{lem}
%%%%%%%%%%%%%%%%%%%%%%%%%%%%%%%%%%%%%%%%%%%%%%
The lemma below is shown similarly as in the proof of \cite[Proposition 4.5]{Dao Takahashi dimension}.
%%%%%%%%%%%%%%%%%%%%%%%%%%%%%%%%%%%%%%%%%%%%%%
\begin{lem} \label{lem for 4.5}
Let $X,Y$ be $R$-modules and $F$ a minimal free resolution of $X$. 
For any $R$-module $M$ and nonnegative integer $i$, we set $\mathfrak{a}_M ^i = \Ann \Ext_R ^i(Y,M)$. 
Then for all $i, j \geq 0$, one has 
\begin{center}
$\mathfrak{a}_X ^i \supseteq \mathfrak{a}_{\Omega_R ^j X} ^{i+j} \mathfrak{a}_{F_0} ^i \cdots \mathfrak{a}_{F_{j-1}} ^i$.
\end{center}
\end{lem}
\begin{proof}
Fix $l\geq 0$. Since $F$ is a minimal free resolution of $X$, there exists an exact sequence
\begin{center}
$0 \to \Omega_R ^{l+1} X \to F_l \to \Omega_R ^l X \to 0$
\end{center}
of $R$-modules. Applying the functor $\Ext_R (Y,-)$, one has an exact sequence
\begin{center}
$\Ext_R ^i(Y, F_l) \to \Ext_R ^i(Y, \Omega_R ^l X) \to \Ext_R ^{i+1}(Y, \Omega_R ^{l+1} X)$.
\end{center}
This yields an inclusion
$\mathfrak{a}_{\Omega_R ^l X} \supseteq \mathfrak{a}_{\Omega_R ^{l+1} X} ^{i+1} \mathfrak{a}_{F_l} ^i$ of ideals. 
Take $l=0, \cdots ,j-1$, then we have $\mathfrak{a}_X ^i \supseteq \mathfrak{a}_{\Omega_R ^j X} ^{i+j} \mathfrak{a}_{F_0} ^i \cdots \mathfrak{a}_{F_{j-1}} ^i$.
\end{proof}
%%%%%%%%%%%%%%%%%%%%%%%%%%%%%%%%%%%%%%%%%%%%%%
The next proposition says that the annihilations of $\Tor$ and $\Ext$ are controlled by the annihilations of $\Tor$ and $\Ext$ on the syzygies of modules of finite length. 
The restriction to the case where R is Cohen--Macaulay recovers \cite[Proposition 4.5]{Dao Takahashi dimension}.
%%%%%%%%%%%%%%%%%%%%%%%%%%%%%%%%%%%%%%%%%%%%%%
\begin{prop} \label{4.5}
Let $R$ be a $d$-dimensional local ring and $n,m$ be nonnegative integers.
\begin{enumerate}[\rm(1)]
\item \label{4.5 tor}
If $a \in \Ann \Tor( \Omega_R ^m (\fl R), \Omega_R ^n (\fl R))$, then one has $a^{2^{2d}}\Tor_i ^R (\mod R, \mod R) = 0$ for all $i>m+n+2d$.
\item \label{4.5 ext}
If $a \in \Ann \Ext( \Omega_R ^m (\fl R), \Omega_R ^n (\fl R))$, then one has $a^{2^{2d}(n+1)}\Ext_R ^i (\mod R, \mod R) = 0$ for all $i>m$.
\end{enumerate}
\end{prop}
\begin{proof}
(\ref{4.5 tor}) 
Fix $i>m+n+2d$. 
Then for all $i-2d \leq l \leq i$, we have $a\Tor_l ^R(\fl R, \fl R) = a\Tor_{l-(m+n)} ^R(\Omega_R ^m (\fl R), \Omega_R ^n (\fl R)) = 0$. 
Hence by Lemma \ref{4.4}, one has $a^{2^{2d}}\Tor_i ^R (\mod R, \mod R) = 0$.
\\
(\ref{4.5 ext})
Let $M,N$ be $R$-modules of finite length. 
Take a minimal free resolution $F$ of $N$. 
By Lemma \ref{lem for 4.5}, we have the inclusion
\begin{equation} \label{4.5 inclusion}
\Ann \Ext_R ^l(\Omega_R ^m M, N) \supseteq
\Ann\Ext_R ^{l+n}(\Omega_R ^m M, \Omega_R ^n N) 
\displaystyle \prod_{k=0} ^{n-1} \Ann\Ext_R ^l (\Omega_R ^m M, F_k) 
\end{equation}
of ideals for all $l>0$. 
Since $\Omega_R ^n N, F_0, \cdots, F_{n-1}$ belong to $\Omega_R ^n (\fl R)$, the right-hand side of \ref{4.5 inclusion} has $a^{n+1}$.
Hence, we have $a^{n+1}\Ext(\Omega_R ^m M, N)=0$. 
Fix $i>m$.
Then for all $i\leq l \leq i+2d$, one has
$a^{n+1}\Ext_R ^l (M, N)=a^{n+1}\Ext_R ^{l-m} (\Omega_R ^m M, N)=0 $.
Thus by Lemma \ref{4.4}, we have $a^{2^{2d}(n+1)}\Ext_R ^i (\mod R, \mod R) = 0$.
\end{proof}
%%%%%%%%%%%%%%%%%%%%%%%%%%%%%%%%%%%%%%%%%%%%%%
As a consequence of the above proposition, the singular locus of a local ring $R$ is contained in closed subsets defined by the annihilators of $\Tor$ and $\Ext$ on the syzygies of $R$-modules of finite length. 
Restricting this result to the case where R is Cohen--Macaulay, we can recover \cite[Proposition 4.6]{Dao Takahashi dimension}.
%%%%%%%%%%%%%%%%%%%%%%%%%%%%%%%%%%%%%%%%%%%%%%
\begin{prop} \label{4.6}
Let $R$ be a $d$-dimensional local ring.
Then the following inclusion relation holds.
\begin{center}
$\Sing(R) \subseteq 
\left( \displaystyle\bigcap_{m,n\geq 0} \V (\Ann \Tor( \Omega_R ^m (\fl R), \Omega_R ^n (\fl R))) \right)
\cap
\left( \displaystyle\bigcap_{m,n\geq 0} \V (\Ann \Ext( \Omega_R ^m (\fl R), \Omega_R ^n (\fl R))) \right).
$
\end{center}
\end{prop}
\begin{proof}
Let $m,n$ be nonnegative integers. 
We shall show that 
$\Sing(R) \subseteq \V (\Ann \Tor( \Omega_R ^m (\fl R), \Omega_R ^n (\fl R)))$. 
Let $\p$ be any prime ideal of $\Sing(R)$ and $a$ an element of $ \Ann \Tor( \Omega_R ^m (\fl R), \Omega_R ^n (\fl R))$.
Then by Proposition \ref{4.5}, we have $a^{2^{2d}}\Tor_i ^R (\mod R, \mod R) = 0$ for all $i>m+n+2d$. 
Hence one has 
$0= a^{2^{2d}}\Tor_i ^R (R/\p, R/\p)_{\p} \cong 
a^{2^{2d}}\Tor_i ^{R_{\p}} (\kappa (\p), \kappa (\p)) $ for all $i>m+n+2d$.
If $a$ does not belong to $\p$, 
one has $\Tor_i ^{R_{\p}} (\kappa (\p), \kappa (\p)) = 0$ for all $i>m+n+2d$.
This implies that $R_{\p}$ is regular, and it contradicts the choice of $\p$. 
Hence, we have $a\in \p$. The assertion for $\Ext$ is also proved similarly.
\end{proof}
%%%%%%%%%%%%%%%%%%%%%%%%%%%%%%%%%%%%%%%%%%%%%%
\begin{rmk}
Let $R$ be as in Proposition \ref{4.6}.
Fix an integer $n\geq0$. 
Since $R$ is noetherian, the ascending chain of ideals
\begin{center}
$\Ann \Tor(\fl R, \Omega_R ^n (\fl R)) 
\subseteq \Ann \Tor( \Omega_R  (\fl R), \Omega_R ^n (\fl R))
\subseteq \Ann \Tor( \Omega_R ^2 (\fl R), \Omega_R ^n (\fl R))
\subseteq \cdots $ 
\end{center}
of $R$ stabilizes. 
Hence so does the descending chain of closed subsets 
\begin{center}
$\V(\Ann \Tor(\fl R, \Omega_R ^n (\fl R)))
\supseteq \V(\Ann \Tor( \Omega_R  (\fl R), \Omega_R ^n (\fl R)))
\supseteq \V(\Ann \Tor( \Omega_R ^2 (\fl R), \Omega_R ^n (\fl R)))
\supseteq \cdots $.
\end{center}
of $\spec R$. 
Thus, there exists an integer $s_n \geq 0$ such that 
\begin{center}
$\displaystyle\bigcap_{m\geq 0} \V (\Ann \Tor( \Omega_R ^m (\fl R), \Omega_R ^n (\fl R))) 
= \V (\Ann \Tor( \Omega_R ^{s_n} (\fl R), \Omega_R ^n (\fl R)))$.
\end{center}
The equalities of ideals and closed subsets for $\Tor(\Omega_R ^n (\fl R),-)$ and $\Ext(-, \Omega_R ^n (\fl R))$ also hold. \par
Let $\X, \Y$ be subcategories of $\mod R$. 
Assume that $\Y$ contains $0$. 
Then by Lemma \ref{lem for 4.5}, we have 
$\Ann \Ext(\X,\Omega_R ^m \Y) \subseteq \sqrt{\Ann \Ext(\X,\Y)}$ 
for all $m \geq 0$. 
Thus, we have an ascending chain of closed subsets 
\begin{center}
$\V (\Ann \Ext(\X,\Y)) 
\subseteq \V (\Ann \Ext(\X,\Omega_R \Y)) 
\subseteq \V (\Ann \Ext(\X,\Omega_R ^2 \Y)) 
\subseteq \cdots$
\end{center}
of $\spec R$. In particular, one has 
\begin{center}
$\displaystyle\bigcap_{m\geq 0} \V (\Ann \Ext( \Omega_R ^n (\fl R), \Omega_R ^m (\fl R))) 
= \V (\Ann \Ext( \Omega_R ^n (\fl R), \fl R))$.
\end{center}
\end{rmk}
%%%%%%%%%%%%%%%%%%%%%%%%%%%%%%%%%%%%%%%%%%%%%%
Now we have reached the main result of this section.
%%%%%%%%%%%%%%%%%%%%%%%%%%%%%%%%%%%%%%%%%%%%%%
\begin{thm} \label{main thm}
Let $R$ be a local ring with residue field $k$ and $\X, \Y$ be subcategories of $\mod R$ with $\radius(\X,\Y) < \infty$. Assume that $\X$ is closed under extensions and contains $\Omega_R ^n k$ for some $n\geq0$.  Then the following inclusion relations hold. 
\begin{itemize}
\item
$\Sing(R) \subseteq \V(\Ann \Tor(\X,\Y)) \subseteq \NF(\Y)$.
\item
$\Sing(R) \subseteq \V(\Ann \Ext(\X,\Y)) \subseteq \NF(\Y)$.
\end{itemize}
\end{thm}
\begin{proof}
Let $m = \radius(\X,\Y)$. 
Then there exists an $R$-module $G$ such that 
$\X \subseteq [G]_{m+1} \subseteq \Y$. 
Since $\X$ contains $\Omega_R ^n k$ and is closed under extensions, we have
$\Omega_R ^n (\fl R) \subseteq \add_R \X \subseteq [G]_{m+1} \subseteq \Y$. 
Hence the following inclusion relations hold.
\begin{align*}
\Sing(R)
& \subseteq \V(\Ann \Tor(\Omega_R ^n (\fl R),\Omega_R ^n (\fl R)))
\text{\quad (by Proposition \ref{4.6})} \\
& \subseteq \V(\Ann \Tor(\add_R \X,\Y)) \\
& = \V(\Ann \Tor(\X,\Y)) \\
& \subseteq \V(\Ann \Tor([G]_{m+1},\mod R))  \\
& = \V(\Ann \Tor(G,\mod R)) 
\text{\quad (by \cite[Lemma 5.3]{Dao Takahashi dimension})} \\
& = \NF(G) 
\text{\quad (by \cite[Proposition 5.1]{Dao Takahashi dimension})} \\
& = \NF([G]_{m+1}) \\
& \subseteq \NF(\Y).
\end{align*}
The assertion for $\Ext$ is proved similarly.
\end{proof}
%%%%%%%%%%%%%%%%%%%%%%%%%%%%%%%%%%%%%%%%%%%%%%
\begin{rmk}
If $R$ is a quasi-excellent local ring, $\mod R$ has a strong generator by \cite[Corollary C]{Dey Lank Takahashi}. 
Hence there exist integers $s,n \geq0$ and an $R$-module $G$ such that $\Omega_R ^s (\mod R) \subseteq [G]_n$. 
In addition, if $R$ satisfies the conditions $(\S_s)$ and $(\G_{s-1})$, then $\Omega_R ^s (\mod R)$ is closed under extensions by \cite[Theorem 2.3]{Matsui Takahashi Tsuchiya} (cf. \cite[Theorem 3.8]{Evans Griffith}). 
Here $(\S_s)$ is Serre's condition, and $(\G_{s-1})$ is a certain condition on the local Gorenstein property of $R$.
Thus, one has $\Sing(R) \subseteq \NF(G)$ by Theorem \ref{main thm}.
\end{rmk}
%%%%%%%%%%%%%%%%%%%%%%%%%%%%%%%%%%%%%%%%%%%%%%
Let $R$ be a $d$-dimensional local ring.
Note that all the closed subsets 
$\V(\Ann \Ext(\C_0(R), \C_0(R)))$, $\V(\Ann \Ext(\C_n(R), \C_n(R)))$, $\V(\Ann \Ext(\Deep_n(R), \Deep_n(R)))$, and $\V(\Ann \Ext(\C(R), \C(R)))$
coincide for all $0\leq n\leq d$ by \cite[Proposition 2.4 (2)]{Kimura}, and \cite[Lemma 3.8]{Dey Takahashi}.
%%%%%%%%%%%%%%%%%%%%%%%%%%%%%%%%%%%%%%%%%%%%%%
\begin{cor} \label{cor1 of thm}
Let $n$ be a nonnegative integer. 
Assume that  $\radius(\C_0 (R), \mod_n R) < \infty$. 
Then the dimensions of
$R/\Ann \Tor(\C_n(R), \C_n(R))$, 
$R/\Ann \Ext(\C_n(R), \C_n(R))$, and
$\Sing(R)$ 
are less than or equal to $n$.
\end{cor}
\begin{proof}
We set $t=\depth R$, and $W_n =\{\p \in \spec R \mid \dim R/\p \leq n\}$. 
Since $\C_0(R)$ contains $\Omega_R ^t k$ and is closed under extensions, we have
$\V(\Ann \Tor(\C_n(R), \C_n(R))) 
\subseteq \V(\Ann \Tor(\C_n(R), \mod_n(R)))
\subseteq \NF(\mod_n R) 
\subseteq W_n $
by Theorem \ref{main thm}. 
Hence one has $\dim R/\Ann \Tor(\C_n(R), \C_n(R)) \leq n$. 
The assertions for the dimensions of 
$R/\Ann \Ext(\C_n(R), \C_n(R))$, and 
$\Sing(R)$ are proved similarly.
\end{proof}
%%%%%%%%%%%%%%%%%%%%%%%%%%%%%%%%%%%%%%%%%%%%%%
Denote by $\assh(R)$ the set of prime ideals $\p$ of $R$ such that the equality $\dim R/\p = \dim R$ holds. Note that $\assh(R)$ is non-empty. 
The next corollary implies that there are few resolving subcategories of finite dimension which contained in $\mod_{d-1} R$ when $R$ is non-Cohen--Macaulay.
%%%%%%%%%%%%%%%%%%%%%%%%%%%%%%%%%%%%%%%%%%%%%%
\begin{cor} \label{cor2 of thm}
Let $R$ be a $d$-dimensioal local ring. 
Assume that $R$ is not Cohen--Macaulay. 
Then we have
$\radius(\C_0(R), \mod_{d-1} R) = \infty$.
In particular, the dimensions of $\C_i(R)$, and $\Deep_i(R)$ are infinite for all $0\leq i \leq d-1$. 
\end{cor}
\begin{proof}
Since $R$ is not Cohen--Macaulay, one has 
$\assh(R) \subseteq \V(\Ann \Ext(\C_{d-1}(R), \C_{d-1}(R))) $ 
by \cite[Proposition 2.6]{Kimura}. 
Take an element $\p$ of $\assh(R)$. 
Then $\V(\Ann \Ext(\C_{d-1}(R), \C_{d-1}(R)))$ contains $\p$ and is specialization-closed. 
Thus, we have $\V(\p) \subseteq \V(\Ann \Ext(\C_{d-1}(R), \C_{d-1}(R)))$.
This implies that $\dim R/\Ann \Ext(\C_{d-1}(R), \C_{d-1}(R)) = d$. 
Hence by Corollary \ref{cor1 of thm}, one has
$\radius(\C_0(R), \mod_{d-1} R) = \infty$. \par
Fix $0\leq i \leq d-1$. 
Since the inequalities 
$\dim \C_i(R),\dim \Deep_i(R) \geq \radius(\C_0(R), \mod_{d-1} R)$
hold, the assertion follows from the above consequence.
\end{proof}
%%%%%%%%%%%%%%%%%%%%%%%%%%%%%%%%%%%%%%%%%%%%%%
\begin{rmk} \label{rmk of cor2}
Let $n$ be a nonnegative integer. 
From the viewpoint of \cite[Theorem 1.2]{Dey Takahashi}, the following four conditions arise for a local ring $R$.
\begin{enumerate}[\rm (i)]
\item
The dimension of $\C_n(R)$ is finite.
\item
The dimension of $R/\Ann \Ext(\C_n(R), \C_n(R))$ is less than or equal to $n$.
\item
The dimension of $R/\Ann \Tor(\C_n(R), \C_n(R))$ is less than or equal to $n$.
\item
The dimension of $\Sing(R)$ is less than or equal to $n$.
\end{enumerate}
As a consequence of the above theorem, the implications $\rm (i) \Rightarrow \rm (ii) \Rightarrow \rm (iii) \Rightarrow \rm (iv)$ hold, and its proof is as follows.
\par
 We set $W_n =\{\p \in \spec R \mid \dim R/\p \leq n\}$ as in the proof of Corollary \ref{cor1 of thm}. \par
$\rm (i) \Rightarrow \rm (ii)$
Since one has
$\dim \C_n(R) = \radius(\C_n(R), \C_n(R)) \geq \radius(\C_0(R), \mod_n R)$, 
we have 
$\dim(R/\Ann \Ext(\C_n(R), \C_n(R))) \leq n$
by Corollary \ref{cor1 of thm}. \par
$\rm (ii) \Rightarrow \rm (iii)$
Since $\C_n(R)$ is closed under syzygies, one has the inclusion
$\Ann \Ext(\C_n(R), \C_n(R)) \subseteq \Ann \Tor(\C_n(R), \C_n(R))$ 
of ideals by \cite[Proposition 3.9]{Dey Takahashi}.
Thus, the assertion holds. \par
$\rm (iii) \Rightarrow \rm (iv)$
Let $t=\depth R$. 
Since one has $\Omega_R ^t (\fl R) \subseteq \C_0(R) \subseteq \C_n(R) $, we have
$\Sing(R) 
\subseteq \V (\Ann \Tor( \Omega_R ^t (\fl R), \Omega_R ^t (\fl R))) 
\subseteq \V (\Ann \Tor( \C_n(R), \C_n(R))  
\subseteq W_n
$.
The first inclusion follows from Proposition \ref{4.6}. 
Hence one has 
$\dim \Sing(R) \leq n$. We are done.
\par
Here we state some remarks for the above four conditions.
\begin{enumerate}[\rm (1)]
\item
Under the assumption of the above condition $\rm(iv)$, one has 
$\C_n(R) = \C(R)$. 
Indeed, let $W_n$ be as in the proof of Corollary \ref{cor1 of thm}. 
Then we have 
$\NF(\C(R)) \subseteq \Sing(R) \subseteq W_n$
by \cite[Remark 3.3 (4)]{Takahashi 2023} and the hypothesis.
Hence the inclusion 
$\C(R) \subseteq \mod_n R $
holds.
Thus we obtain $\C_n(R) = \C(R)$.
\item
If $R$ is a quasi-excellent Cohen--Macaulay local ring, then the implication 
$\rm (iv) \Rightarrow \rm (ii)$
in the above conditions holds. 
Indeed, since $R$ is quasi-excellent, the category $\mod R$ has a strong generator by \cite[Corollary C]{Dey Lank Takahashi}, and the equality $\Sing(R) = \V(\ca(R))$ holds by \cite[Theorem 4.3]{Iyengar Takahashi}. Here $\ca(R)$ is the cohomology annihilator of $R$, that is, defined as follows:
\begin{center}
$\ca(R)= 
\displaystyle\bigcup_{n\geq 0} 
\left( \displaystyle\bigcap_{m\geq n} \Ann \Ext_R ^m(\mod R, \mod R) \right) $.
\end{center}
Thus we have 
$\Sing(R) = \V(\ca(R)) = \V(\Ann \Ext(\cm(R), \cm(R)))$
by \cite[Proposition 2.4 (3)]{Kimura}. 
In addition, if $R$ admits a canonical module, then the implication
$\rm (iv) \Rightarrow \rm (i)$
holds by \cite[Proposition 4.5.1 (2)]{Dey Lank Takahashi}.
\item
Let $d = \dim R$. 
If $R$ is not a Cohen--Macaulay ring, then we have 
$\dim \C_n(R) = \infty$ for all $0\leq n\leq d-1$ 
by Corollary \ref{cor2 of thm}. 
On the other hand, in the case of $n=d$, the above conditions (ii), (iii), and (iv) are trivial. 
Therefore, the implications 
$\rm (i) \Rightarrow \rm (ii) \Rightarrow \rm (iii) \Rightarrow \rm (iv)$
have no meaning for non-Cohen--Macaulay local ring.
\end{enumerate}
\end{rmk}
%%%%%%%%%%%%%%%%%%%%%%%%%%%%%%%%%%%%%%%%%%%%%%
Here is an application example of the above corollary for a non-Cohen--Macaulay local ring.
%%%%%%%%%%%%%%%%%%%%%%%%%%%%%%%%%%%%%%%%%%%%%%
\begin{eg} \label{eg of cor2}
Let $K$ be a field and $R= K\llbracket x, y, z \rrbracket /(x^2, xy)$. Then $R$ is dimension $2$ and $\depth$ $1$, so is not Cohen--Macaulay. 
Hence, we have $\dim \C_0(R) =\dim \C_1(R) =\dim \Deep_1(R) = \infty$ by Corollary \ref{cor2 of thm}.
\end{eg}
%%%%%%%%%%%%%%%%%%%%%%%%%%%%%%%%%%%%%%%%%%%%%%
The above corollary does not deal with the case where $i=d$, that is, on the finiteness of dimensions of $\C(R)$ and $\Deep(R)$. Therefore, the following question arises.
%%%%%%%%%%%%%%%%%%%%%%%%%%%%%%%%%%%%%%%%%%%%%%
\begin{ques}
Let $R$ be a noetherian local ring. If $R$ is not Cohen--Macaulay, then are the dimensions of the subcategories $\C (R)$ and $\Deep(R)$ infinite?
\end{ques}
%%%%%%%%%%%%%%%%%%%%%%%%%%%%%%%%%%%%%%%%%%%%%%%%%%%%%%%%%%%%
Let $d=\dim R$. 
If $\depth R =0$, then we have
$\dim \Deep(R) = \dim (\mod R) = \radius(\mod R) = \infty$ 
by \cite[Theorem 4.4]{Dao Takahashi radius}. 
On the other hand, assume that the dimension of $\Sing(R)$ is less than $d$ (e.g. $R$ is reduced), and we set $n=\dim \Sing(R)$. 
Then one has 
$\dim \C(R) = \dim \C_n(R) = \infty$ 
by Corollary \ref{cor2 of thm}. 
In general case, we do not know any examples where $\C(R)$ or $\Deep(R)$ has finite dimension.
%%%%%%%%%%%%%%%%%%%%%%%%%%%%%%%%%%%%%%%%%%%%%%%%%%%%%%%%%%%%
\begin{ac}
The author would like to thank his supervisor Ryo Takahashi for giving many thoughtful questions and helpful discussions. He also thanks Yuya Otake and Kaito Kimura for their valuable comments.
\end{ac}
%%%%%%%%%%%%%%%%%%%%%%%%%%%%%%%%%%%%%%%%%%%%%%%%%%%%%%%%%%%%%
%%%%%%%%%%%%%%%%%%%%%%%%%%%%%%%%%%%%%%%%%%%%%%%%%%%%%%%%%%%%

\end{document}